%% file: main.tex
\documentclass[runningheads]{llncs}
\usepackage{lipsum}
\usepackage{amsmath,amsfonts}
\usepackage{graphicx}
\usepackage{epstopdf}
\usepackage{algorithmic}
\usepackage{tikz,pgfplots,pgfplotstable}
\pgfplotsset{compat=1.16}
\usepackage{colortbl}
\usepgfplotslibrary{statistics}
\usetikzlibrary{pgfplots.statistics}
\usepackage{booktabs}
\usepackage{multirow}
\usepackage{array}
\usepackage[algo2e,ruled,linesnumbered]{algorithm2e}
\usepackage{makecell}

\newdimen\fwd

\newcommand{\Id}{{\mathrm{Id}}}
\newcommand{\LSy}{{\mathrm{LSy}}}
\newcommand{\LSp}{{\mathrm{LSp}}}
\newcommand{\Dy}{{\mathrm{Dp}}}
\newcommand{\Dp}{\Dy}
\newcommand{\Dz}{{\mathrm{Dz}}}
\newcommand{\Du}{{\mathrm{Du}}}
\newcommand{\GM}{{\mathrm{GM}}}

\newcommand{\FAIR}{{\mathrm{FAIR}}}

\newcommand{\LBFGS}{\mbox{$\ell$-BFGS}}

\newcommand{\ie}{\emph{i.e.}}
\newcommand{\eg}{\emph{e.g.}}

\newcommand{\R}{\mathbb{R}}
\newcommand{\eps}{\varepsilon}

\begin{document}
\title{Hessian Initialization Strategies for \LBFGS~Solving Non-linear Inverse Problems} 
\titlerunning{Hessian Initialization Strategies for \LBFGS}
%

\author{Hari Om Aggrawal \inst{1}
\and Jan Modersitzki \inst{1,2}} 
\authorrunning{H. O. Aggrawal et. al.}
\institute{Institute of Mathematics and Image Computing,
	University of L\"ubeck, Germany \and Fraunhofer Institute for Digital Medicine MEVIS, L\"ubeck, Germany \\
	\email{hariom85@gmail.com, modersitzki@mic.uni-luebeck.de}
}

\maketitle              
\begin{abstract}
	\LBFGS~is the state-of-the-art optimization method for many large scale inverse problems. It has a small memory footprint and achieves superlinear convergence. 
	The method approximates Hessian based on an initial approximation and an update rule that models current local curvature information. The initial approximation greatly affects the scaling of a search direction and the overall convergence of the method. 
	
	We propose a novel, simple, and effective way to initialize the Hessian. 
	Typically, the objective function is a sum of a data-fidelity term and a regularizer.
	Often, the Hessian of the data-fidelity is computationally challenging, but the regularizer's Hessian is easy to compute.
	We replace the Hessian of the data-fidelity with a scalar and keep the Hessian of the regularizer to initialize the Hessian approximation at every iteration. The scalar satisfies the secant equation in the sense of ordinary and total least squares and geometric mean regression.
	
	Our new strategy not only leads to faster convergence, but the quality of the numerical solutions is generally superior to simple scaling based strategies. Specifically, the proposed schemes based on ordinary least squares formulation and geometric mean regression outperform the state-of-the-art schemes.
	
	The implementation of our strategy requires only a small change of a standard \LBFGS~code. 
	Our experiments on convex quadratic problems and non-convex image registration problems confirm the effectiveness of the proposed approach. 

\end{abstract}
%
%
%

\input{01-introduction}
\input{02-initmethods}

\input{03-experiments}

\input{04-conclusion}

\bibliographystyle{splncs04}
\bibliography{refs}

\end{document}

%% file: 01-introduction.tex
\begin{section}{Introduction}

Many real-life problems fit the framework of an inverse problem. Fluorescence optical tomography \cite{Patil2019}, ultrasound tomography \cite{Bernhardt2020}, and photoacoustic tomography \cite{Saratoon2013} are just a few non-invasive imaging techniques that image a human body's internal structure by solving inverse problems. 

Inverse problems are typically ill-posed in nature \cite{Hansen2010}. The solution may not be unique and unstable with variations in the data due to unavoidable factors such as physical noise. Regularizing the problem with prior information, we obtain a solution by minimizing an objective function
\[
J:\R^n\to\R,\quad J(x)=D(x)+S(x),
\]
where $D$ denotes a data-fitting term and $S$ a regularizer. For many non-linear problems, the objective function is non-convex, and the main limitation is computationally demanding operations. Hence, an efficient optimization method to be designed that requires fewer evaluations of an objective function, its gradient and Hessian, and more occasional calls to a linear solver.

Numerous optimization schemes exist to solve these problems. Still, schemes that do not require more than first-order information are generally preferable. Hessian computation is usually expensive. 

Steepest-descent (SD) and quasi-newton methods are the most popular first-order methods. 
SD converges only linearly; hence super-linear convergent quasi-Newton methods such as Gauss-Newton (GN) schemes or the Broyden-class are preferable. 
The quasi-Newton method's key idea is to replace Hessian with an approximation that models the local curvature information. It leads to not only faster convergence but as well higher solution accuracy than simple gradient descent methods; see, e.g. \cite{Patil2019,Heldmann2006} and references therein.

The Hessian approximation in the GN method is based on linearizing a function that involves a matrix-vector product with a Jacobian matrix. For applications such as optical tomography \cite{Patil2019}, the Jacobians are generally dense, and hence per
iteration costs can be very high. Therefore, Broyden-class methods are preferred in practice. The most popular member is the limited-memory version of BFGS scheme (\LBFGS) for large scale inverse problems; see its application to recent work in ultrasound tomography \cite{Bernhardt2020} and image registration \cite{Knig2018}.

The Broyden-class works with approximations of the Hessian that are based on an initial approximation and an update rule that is typically based on current curvature information derived from a secant equation \eqref{eq:secant}. Setting $x=x_k$, $x'=x_{k+1}$, $p=x'-x$ and introduce $y:=\nabla J(x')-\nabla J(x)$, a Taylor expansion
\(
y=\nabla J(x')-\nabla J(x)\approx\nabla^2J(x)p
\)
motivates the so-called secant-equations~\cite[Chapter 2]{Nocedal2006} 
for Hessian $B$,
\begin{equation}\label{eq:secant}
B'\>p=y
\quad\mbox{or, for the inverse of $B$,}\quad	 
p=H'\>y.
\end{equation}

Based on an initial choice $H_0$, BFGS-schemes update the current approximation $H=H_k$ using a constrained and weighted least squares fit,
\[
H_{k+1}\in
\mathrm{argmin}\{|M-H_k|_F,\ M=M^{\top},\ \quad My_k=p_k\},
\]
where a weighted Frobenius-norm $|A|_F^2=\mathrm{trace}(WAWA^{\top})$ is used.
If the weight matrix satisfies secant equation $Wp=y$, one obtains a unique and scale-invariant solution for $H_{k+1}$ as a rank two update of the $H_k$,
\begin{equation}\label{eq:recurrence}
H_{k+1}= V_k^{\top}H_kV_k+ \alpha_k\> p_k p_k^{\top},
\quad\mbox{with}\quad
\alpha_k:=(y_k^{\top}p_k)^{-1},\quad
V_k:=I-\alpha\> y_k p_k^{\top}.
\end{equation}

For large-scale problems, a limited-memory version of BFGS (\LBFGS) is
used~\cite{Liu1989}. In \LBFGS, at most the last $\ell$ pairs are used.
More precisely, only pairs $(y_j,p_j)$ with 
$k_\ell:=\max\{1,k-\ell-1\}\le j\le k$ are used.
Formally, $H_{k+1}=M_{k+1}$ results from the modified recursion 
\begin{equation}\label{eq:lBFGSrecurrence}
M_{j+1}:=V_{j}^{\top}M_{j}V_j+\alpha_j\>p_jp_j^{\top},\quad 	
j=k_\ell,\ldots,k.
\end{equation}


The convergence depends on the quality of Hessian approximation which generally can
not be controlled. It has been observed numerically that a ``good'' initial
guess of the Hessian greatly affects the scaling of a search direction and convergence of the overall
scheme~\cite{Dener2019,Gilbert1989,Marjugi2013,Andrei2020}. Note that \LBFGS-method allows to re-initialize Hessian at every iteration. This opportunity provides a window to rescale the search direction and infuse more
information in the scheme.

The state-of-the-art strategy initialize the Hessian (or it's inverse) with a scaled identity matrix
\[
	H_0^k = \tau_kI.
\]
The scalar $\tau_k$ is computed at each iteration to satisfy the secant equation \eqref{eq:secant} in a ordinary least
square sense following the Oren–Luenberger scaling strategy \cite{Oren1982}. This results in two choices for scaling factor $\tau_k$, \ie,
\begin{equation}\label{eq:LSyp}
	\tau^\LSy_k=(y^{\top}p)/(y^{\top}y) \quad \mbox{and} \quad \tau^\LSp_k=(p^{\top}p)/(y^{\top}p).
\end{equation}
In practice, it has been observed that the factor $\tau_k^\LSy$ ensures a well-scaled search
direction and as a result, most of the iterations accept a steplength of
one~\cite{Gilbert1989}.

In this paper, we suggest to improve the quality of the initial Hessian
approximation by including computationally manageable parts from the
regularizer. More precisely, we suggest to use
\[
	B_0^k=\tau_kI+\nabla^2S,
\]
where $\nabla^2S$ denotes the Hessian of a not necessarily quadratic
regularizer. We derive four options for scaling factor $\tau_k$ based on ordinary
and total-least squares formulations, and geometric mean regression.

Our work is motivated by ideas in image registration~\cite{2009-FAIR,Heldmann2006} and molecular energy minimization \cite{Jiang2004}.
In~\cite{2009-FAIR,Heldmann2006}, the Hessian is initialized by a
positive-definite matrix $B_0=\tau I+A$, where $A$ is the Hessian of a
quadratic regularizer and constant. The parameter $\tau>0$ is chosen manually.
In \cite{Heldmann2006}, it is reported that this strategy outperforms the
simple scaling approach. 
In \cite{Jiang2004}, the proposed strategy is similar to ours but requires expensive incomplete Cholesky factorization at each iteration to ensure positive-definiteness of Hessian approximation. 
Moreover, the value for $\tau$ is heuristically defined. But, in this paper, we show that it satisfies secant equation in a sense of geometric mean regression.

We assume the regularization part to be computationally manageable. Typical
examples include $L_2$-norm based Tikhonov regularizers
\cite{2009-FAIR,Hansen2010}, smooth total-variation norm \cite{Vogel2002}, or,
more generally, quadratic forms of derivative based regularization. Here,
$R(x)=\|Bx\|_{L_2}$ and $B$ is a linear differential operator.
Non-quadratic forms such as the hyperelastic regularizer~\cite{Burger2013} also
fit into this class. 

Our new strategy is easy to integrate into an \LBFGS~code. Only the Hessian
initialization routine needs to be changed, all other parts remain unchanged.

In this paper, we also demonstrate on various test cases that the proposed approach achieves fast convergence and improves the solution accuracy compared to the standard scaling based approaches. Our test cases include convex quadratic problems and non-convex image registration problems with both, quadratic and non-quadratic regularization. Due to the page limitation, the theoretical investigation will be a part of the extended version of this paper.

In Sec.~\ref{Sec:newInitMethod}, we derive four scaling factors for the proposed initialization strategy and present a practical algorithm. We report the numerical experiments with results in Sec.~\ref{Sec:experiments}. In Sec.~\ref{Sec:conclusion}, we conclude our findings.

\end{section}

%% file: 02-initmethods.tex
\begin{section}
	{Proposed Hessian Initialization Strategy and Algorithm}%
	\label{Sec:newInitMethod}

As common for inverse problems, we assume that the objective function $J$ is a
sum of a data fitting term $D$ and a regularizer $S$. Hence
\begin{equation}\label{eq:d2J}
	\nabla^2 J = \nabla^2D+\nabla^2S,
\end{equation}
where $A_k:=\nabla^2S(x_k)$ is symmetric positive semidefinite (SPSD). We
assume that $A_k$ is ``easy'', i.e. has low memory requirements and $A_kx=b$
can be solved efficiently. Problems may occur from the data fitting part
$\nabla^2D(x_k)$, which might be computational complex and potentially
ill-conditioned.

In the proposed strategy, we suggest to approximate $\nabla^2D(x_k)\approx
\tau_kI$, where $\tau_k$ is a tuning parameter to be determined. Hence
\begin{equation}\label{B0approx}
	B_0^k = \tau_k I + A_k.
\end{equation}
The role of $B_0^k$ is to mimic the Hessian at least for the current update. 
In this regard, we aim to satisfy the secant equation~\eqref{eq:secant} in some least squares sense, where now
\begin{equation} \label{eq:Bk0Secant}
	y=B_0^k\>p=\tau_k\>p+A_kp
	\iff \tau_kp-z=0,\quad
	z:=y-A_kp.
\end{equation}
Since $B_0^k$ is required to be symmetric positive definite (SPD), we have
$\tau_k\ge\tau_{\min}:=\eps-\mu_{\min}$,
where $\eps>0$ is a small tolerance, typically $\eps=10^{-6}$ and $\mu_{\min}$ is the smallest eigenvalue of $A_k$. This adds a constraint to the least squares problems. First we summarize the ordinary least square approach; cf.~Lem.~\ref{lem:ls}. We removed subscript $k$ for clarity.

\begin{lemma}\label{lem:ls}
	Let $p,z\in\R^n$ with $p^{\top}z\ne0$ and $\tau_{\min}\in\R$. 
	Then
	\begin{equation*}
	\tau^\Dy :=\max\{(p^{\top}z)/(p^{\top}p),\tau_{\min}\}
	\mbox{ and }
	\tau^\Dz :=\max\{(z^{\top}z)/(p^{\top}z),\tau_{\min}\}
	\end{equation*}
	are optimal scaling parameters resulting from a minimization 
	of $|\xi u-v|$ subject to $\xi\ge\tau_{\min}$, where
	$(u,v)=(p,z)$ for (Dp) and $(u,v)=(z,p)$ for (Dz).
	
	Moreover, it holds $|\tau^\Dy|\le|\tau^\Dz|$.
\end{lemma}

\begin{proof}
	For (Dp): the unique minimizer $\xi$ of the unconstrained problem follows from the basic calculus. If $\xi<\tau_{\min}$, the minimum is attained on the
	boundary. For (Dz): the result follows from rescaling.
	The inequality follows the Cauchy-Schwarz-inequality
	$|p^{\top}z|^2\le(p^{\top}p)(z^{\top}z)$.
\end{proof}

The above choices have a preference either for the $p$ or $z$ direction.
A total least squares approach can be used for an unbiased approach;
cf.~Lem.~\ref{lem:tls}.

\begin{lemma}\label{lem:tls}
	Let $p,z\in\R^n$ with $\delta:=p^{\top}z\ne0$, $\tau_{\min}\in\R$, Then
	\[
	\tau^\Du
	:=\max\{(|z|^2-\lambda)/\delta,\tau_{\min}\}, \quad \lambda=(|p|^2+|z|^2-\sqrt{(|p|^2-|z|^2)^2+4\delta^2})/2,
	\]
	is an optimal scaling parameter from the rescaling of minimizer $\eta = [\eta_1,\eta_2]$ of the total least squares formulation $|\eta_1 p-\eta_2z|$ subject to $|\eta|=1$.
	With $\tau^\Dp$ and $\tau^\Dz$ as in  Lem.~\ref{lem:ls}, it holds $|\tau^\Dp|\le|\tau^\Du|\le|\tau^\Dz|$.
\end{lemma}

\begin{proof}
	We have the necessary condition of first order $(U^{\top}U-\mu I)\eta=0$ subject to  $|\eta|=1$, where $U=[p,-z]$ and $\mu$ denotes the Lagrange-multiplier. This indicates that $\mu$ is the smallest eigenvalue of the symmetric 2-by-2 matrix $U^{\top}U$ with diagonal elements $|p|^2$ and $|z|^2$ and off diagonal $-\delta$. Hence, $\mu=\lambda$ and $\eta:=v/|v|$ is a normalized version of the associated eigenvector~$v=(\delta,|p|^2-\lambda)$. The value for $\tau^\Du$ follows from proper scaling.
	
	To show the inequality, we use the relationship $(|z|^2-\lambda)/\delta = \delta/(|p|^2 - \lambda)$ derived from $\det(U^{\top}U-\mu I) = 0$. Since $U^{\top}U$ is SPSD, we know $\lambda\geq0$. Hence, the inequalities $|(|z|^2-\lambda)/\delta| \leq |z|^2/|\delta|$ and $|\delta|/|p|^2 \leq |\delta/(|p|^2 - \lambda)|$ satisfies. It leads to $|\tau^\Dp|\le|\tau^\Du|\le|\tau^\Dz|$ following the definition of $\tau^\Dp$, $\tau^\Dz$, and $\tau^\Du$.
\end{proof}

Geometric mean regression is an another unbiased approach; see \cite{Draper1991} for details. The optimal scaling parameter is defined as the geometric mean of scaling parameters obtained from ordinary least squares problems in Lem.~\ref{lem:ls}, \ie,
\begin{equation}\label{eq:GM}
	\tau^\GM = \max\{(\tau^\Dp\tau^\Dz)^{1/2},\tau_{\min}\} = \max\{(|z|^2/|p|^2)^{1/2},\tau_{\min}\}
\end{equation}
and follows $|\tau^\Dp|\le \tau^\GM \le|\tau^\Dz|$.

\subsubsection{Remarks on scaling parameters:} 
Note that, the tuning of parameter $\tau$ changes both the angle and length of search direction, whereas the simple scaling based schemes $\tau^\LSy$ and $\tau^\LSp$ majorly changes the length of the search direction.

To achieve fast convergence, we aim to reduce the number of iterations and the line-search steps at every iteration. For that, we seek a search direction that is closer to the Newton direction and take fewer iterations to convergence. Furthermore, we seek well-scaled search directions that satisfy steplength equal to one and avoid any line-search steps for reducing the total run-time.

For a simple quadratic problem, we observe that the search directions with the proposed choices for $\tau$ behave almost in a similar fashion with respect to the Newton direction. Hence, all options are practically equivalent. 

But, in practice, we observe that the length of a search direction is inversely proportional to the value of $\tau$ for our schemes.
Hence, a small $\tau$ leads to a long step. Although it is desirable, but an overestimated length leads to many line-search steps. On the other hand, with a large $\tau$, we take small steps and, as results, require many iterations for convergence; see results in Sec.~\ref{Sec:experiments}. These facts suggest an optimal scaling factor, but the exact criterion are so far unknown to us. Nevertheless, we provide four choices for $\tau$ covering a wide range and describe their inter-relationship in Lem.~\ref{lem:tls} and \eqref{eq:GM}.

\subsubsection{Practical algorithm:} Now, we are ready to present pseudo code for the standard \LBFGS~algorithm with the proposed Hessian initialization strategies where we motivate to initialize Hessian $B_0^k$ at every iteration with \eqref{B0approx}; see Algo.~\ref{algo:lbfgs} \cite{Nocedal2006}. In the standard \LBFGS~code, we only need to change the Hessian initialization routine; see Line 4 in Algo.~\ref{algo:lbfgs}; with a few lines of code described in Algo.~\ref{algo:Hinit}.

\begin{algorithm2e}[t]\label{algo:lbfgs}
	\caption{Standard \LBFGS~algorithm with the proposed Hessian initialization strategy} 
	\SetAlgoLined
	Initialize a starting guess $x_0$, integer $\ell > 0$, and $\eps > 0$\;
	$k \longleftarrow 0$\;
	\Repeat{convergence}{
		Compute $B_0^k$ following steps in Algo. \ref{algo:Hinit}\;
		Compute search direction $d_k \longleftarrow - H_k \nabla J(x_k)$ using $B_0^k$ in the two-step recursive algorithm based on \eqref{eq:lBFGSrecurrence}; see details in \cite{Liu1989}\;
		Compute $x_{k + 1} \longleftarrow x_k + \alpha_k d_k$ where $\alpha_k$ is obtained with a line-search algorithm\;
		\If{$k > \ell$}{
			Discard the vector pair \{$p_{k-\ell}$,$y_{k-\ell}$\} from the storage\;
		}
		Compute and save $p_k \longleftarrow x_{k+1} - x_k$ and $y_k = \nabla J(x_{k+1}) - \nabla J(x_k)$\;
		$k \longleftarrow k + 1$\;
	}
\end{algorithm2e}

\begin{algorithm2e}[t]\label{algo:Hinit}
	\caption{Secant equation based Hessian initialization strategies} 
	\SetAlgoLined
	Compute $A_k$, Hessian of regularizer at $x_k$\;
	Set $\tau_{\min} \longleftarrow \eps$\;
	\eIf{first iteration ($k = 0$)}{
		Set $\tau_k \longleftarrow \tau_{\min}$\;
	}{
		Compute $z_k \longleftarrow y_k - A_kp_k$\;
		Set $\tau_k$ to either $\tau_k^\Dy$, $\tau_k^\Du$, $\tau_k^\Dz$, or $\tau_k^\GM$\;		
	}
	Initialize $B_0^{k} \longleftarrow \tau_k I + A_k$\;
\end{algorithm2e}

To initialize $B_0^k$, we start with setting $\tau_{\min}$, computing the Hessian of regularizer, and evaluating $z_k$. The parameter $\tau_k$ can be set with either $\tau_k^\Du$, $\tau_k^\Dy$, $\tau_k^\Dz$, or $\tau_k^\GM$.

In the first iteration, we can not compute $z_k$ due to the lack of information on the required iterates. Hence, initially, we set $\tau = \eps$ in our experiments. Other initialization options, \eg, based on the norm of a gradient \cite{Nocedal2006}, are also possible. 

Recall that the parameter $\tau_k$ should be greater than $\tau_{\min} = \eps - \mu_{\min}(A_k)$ to ensure the positive-definitness of Hessian $B_0^k$. To determine $\tau_{\min}$, we need the smallest eigenvalue of $A_k$ that could be computationally expensive operation for large scale problems. Hence, we avoid the eigenvalue computation in practice and set $\tau_{\min} = \eps = 10^{-6}$ in our experiments.

\end{section}

%% file: 03-experiments.tex
\begin{section}{Numerical Experiments and Results}\label{Sec:experiments}

\newcommand{\ig}[1]{}
\newcommand{\tol}{\varepsilon}
\fwd=18mm

We report on the performance of the proposed Hessian initialization strategies for typical inverse problems:
\textbf{a) Strictly convex quadratic problems:} This class is chosen to
validate the convergence properties of the proposed strategy numerically.
\textbf{b) Non-convex image registration problems:} This class is chosen to show the effectiveness of the proposed strategy on a few challenging real-world problems.

\renewcommand{\arraystretch}{1.2}
\begin{table}[t]
	\caption{Optimization methods used for evaluations. $A_k$ be the Hessian of regularizer.}
	\setlength{\tabcolsep}{2pt}
	\centering
	\begin{tabular}{cccc}
		\toprule
		No. & Optimization methods & Hessian initialization & References\\
		\midrule
		1.   & Steepest descent (SD)   &    -                                                                  &  see \cite{Nocedal2006}   \\
		2.  & \LBFGS         					& $H_k^0= \tau_k I $           			&  $\tau_k = I$, $\tau_k^\LSy$, or $\tau_k^\LSp$; see \eqref{eq:LSyp}    \\
		3.  & \LBFGS~(FAIR scheme)    & $B_k^0~= \tau I + A_k $ & $\tau$ is set manually; see \cite{2009-FAIR} \\
		4.  & \LBFGS~(proposed)        & $B_k^0~= \tau_k I + A_k $   &  
		\makecell{
		$\tau_k = \tau_k^\Dp$, $\tau_k^\Dz$, $\tau_k^\Du$, or $\tau_k^\GM$; \\
		see Lem.~\ref{lem:ls}, Lem.~\ref{lem:tls}, and \eqref{eq:GM}}         \\
		5.  & Gauss-Newton (GN)      &  -                                                                    		& see \cite{Nocedal2006} \\
		\bottomrule    
	\end{tabular}
	\label{tab:optmethods}
\end{table}
\renewcommand{\arraystretch}{1} 

We investigate in total eight Hessian initialization strategies for $\LBFGS$ method; see Tab.~\ref{tab:optmethods} for details. Along with \LBFGS, we also report results with Gauss-Newton (GN) and steepest descent (SD) method. GN is a widely used method in the field of image registration. GN may achieve quadratic convergence close to the solution. Even though, GN may converge to a local optimal point in a few iterations, but for large scale problems, per iteration cost for GN could be very high due to additional matrix-vector products with Jacobian; see run-time for GN in Tab.~\ref{tab:IR_Results} for image registration problems. On the other hand, SD follows a linear and \LBFGS~a superlinear convergence. 

Note that, the methods $\Dy$, $\Du$, $\Dz$, $\GM$, FAIR, and GN solve a linear system at each iteration to compute the search direction. For that, we use Jacobi preconditioned conjugate gradient (PCG) method. Moreover, the associated system matrices are not stored, rather matrix-vector product has been computed directly. We run PCG until the relative residual is less than $10^{-6}$ or the maximum iterations reach to 100. The iteration count is set to low with the purpose of reducing extra computational time at each iteration due to the linear solver. 

The Hessian initialization in the FAIR \cite{2009-FAIR} is similar to ours. But, they set manually the parameter $\tau = 10^{-3}c$, where $c$ is the first diagonal element of $A$.

We use Armijo backtracking line-search algorithm to estimate the step-size. As noted in \cite{Nocedal2006}, if curvature condition is not satisfied at any iteration, we skip the Hessian update. For stopping criteria, we follow \cite[p.~78]{2009-FAIR}  
and set $\tol_J = 10^{-5}$, $\tol_W = 10^{-1}$, and $\tol_G = 10^{-2}$. 
For \LBFGS, we use the standard choice $\ell = 5$. 

Run-time and solution accuracy are our main criteria to evaluate the performance of optimization methods.

\begin{subsection}{Quadratic problem}\label{Sec:ImgDeBlur}

\input{table_quadPrb.tex}

We minimize a strictly convex quadratic function $0.5(x - c)^\top(D + \alpha R)(x - c)$ that has the unique minimizer at $x^{*} = c \in \R^n$ with $D$ and $R$ be a symmetric and positive-definite matrix and $\alpha > 0$. In experiments, $D$ is a diagonal matrix with exponentially decaying eigenvalues, \ie, $D_{ii} = \exp(-i)$. It is a highly ill-conditioned matrix with condition number of order $10^6$ reflecting Hessian of a typical data-fidelity term in inverse problems. The regularization matrix $R$ be a well-known Laplacian matrix with zero boundary conditions. The regularization parameter $\alpha$ controls the ill-conditioning of the quadratic function. Here, we investigate weakly ($\alpha = 10^{-5}$), mildly ($\alpha = 10^{-3}$), and strongly ($\alpha = 10^{-1}$) regularized problems; see results in Tab.~\ref{tab:quadPrb}.

The iterations start with $x$ be a zero vector. The iterations stop when either the relative error $\|x-x^{*}\|/\|x^{*}\| \leq 10^{-5}$ or the iteration count reaches 5000.

As expected, the highly ill-conditioned problem, \ie, weakly regularized, requires many iterations for convergence. Especially, if the local curvature is not well-estimated; see results for $\ell = 1$ in Tab.~\ref{tab:quadPrb}. 
The iteration counts are decreasing with increasing regularization levels and with improving Hessian approximation that is increasing $\ell$.
Note that this behavior is consistent across all Hessian initialization strategies taken into consideration in this work. 

Identity initialized Hessian scheme converges very slow. But mostly, it satisfies step-length equal to one, which means it takes tiny steps at each iteration. 

In most cases, the Hessian initialization schemes equipped with regularization require fewer iterations than the simple scaling based $\LSy$ and $\LSp$ schemes. In particular, the $\FAIR$ scheme takes the lowest iterations, but search-directions are badly scaled. Hence, line-searches (LS) per iteration are much higher than the other schemes. Moreover, LS steps highly depend on the regularization level. 

But, for the proposed four schemes, the LS steps depend on the goodness of Hessian approximation, \ie, the value of $\ell$ rather than the regularization level. In practice, we generally work with a fixed $\ell$ and adjust the regularization level as per the need. Hence, the proposed scheme suits better for such a scenario. In particular, the $\Dz$ scheme generally take 1.15 LS steps per iteration and does not depend much on $\ell$. The $\Dp$ and $\Du$ schemes require higher LS steps than $\Dz$ whereas the $\GM$ between the $\Dp$ and the $\Dz$. In terms of iterations, we observe an almost inverse relationship; \eg, the $\Dp$ and $\Du$ scheme take fewer iterations than $\Dz$; follow the discussion in Sec.~\ref{Sec:newInitMethod} for the underlying reason.

\end{subsection}

\begin{subsection}{Image Registration}

\renewcommand{\arraystretch}{1}
\begin{table}[t]
	\caption{Image registration test problems (TP) for the performance evaluation of optimization strategies. For TP-4, the initial TRE is not available (N.A.).}
	\centering
	\setlength{\tabcolsep}{1.1pt}
	\begin{tabular}{ccccccc}
		\toprule
		TP & Dataset & Problem size & \makecell{Data- \\ Fidelity} & Regularizer & Parameters & Initial TRE\\
		\midrule
		1     & Hand (2D) & $2\times128\times128$	& SSD		& Curvature & $\alpha = 1.5 \times 10^3$ & 1.04 (0.62)\\
		2     & Hand (2D) & $2 \times 128 \times 128$	& MI		& Elastic & $\alpha = 5 \times 10^{-3}$  & 1.04 (0.62)\\
		3     & Lung (3D) & $3 \times 64 \times 64 \times 24$	& NGF		& Curvature & $\alpha = 10^2$ & 3.89 (2.78) \\
		4     & Disc-C (2D) &	$2 \times 16 \times 16$	  & SSD       & Hyperelastic & $\alpha = (100,20)$ & N.A. \\
		\bottomrule    
	\end{tabular}
\label{tab:IRproblems}
\end{table}
\renewcommand{\arraystretch}{1} 

\input{table_imgReg}

Now, we show effectiveness on four real-life large-scale problems from image registration. The registration problems are generally highly non-convex and ill-posed in nature; see \cite{2009-FAIR} for details. Here, given a pair of images $T$ and $R$, the goal is to find a transformation field $\phi$ such that the transformed image $T(\phi)$ is similar to $R$, \ie, $T(\phi) \approx R$. To determine $\phi$, we solve an unconstrained optimization problem 
\[
J(\phi) = D(T(\phi), R) + \alpha S(\phi) \xrightarrow{\phi} \min
\]
where $D$ measures the similarity between the transformed image $T(\phi)$ and $R$. The regularizer $S$ enforces smoothness in the field. Curvature, elastic, and hyperelastic are a few commonly used regularizers.  The typical choices for similarity measures are the sum of squared difference (SSD), normalized gradient fields (NGF), and mutual information (MI). 

Our four test problems (TP) represent a big class of registration models; see Tab.~\ref{tab:IRproblems}. The popular X-ray hand images are from \cite{2009-FAIR}, lung CT images from the well-known DIR dataset \cite{Knig2018,Castillo2009}, and the academic Disc-C images from \cite{Burger2013}.

Note that our strategy works even when the Hessian of regularizer is available only partially. For that, we consider hyperelastic regularizer; see \cite{Burger2013} for details.

The ground truth transformation fields are not available for real-world problems. Hence, we compute the target registration error (TRE), defined as the Euclidean distance between the ground truth landmarks and the estimated landmarks after registration. To accumulate the TRE for each landmark position, we compute the mean and standard deviation (std.) of TRE. The regularization parameter is set to achieve the lowest TRE without foldings in the field.

The field $\phi$ is initialized with an identity map, \ie, $\phi_0(x) = x$ in all experiments. The open-source FAIR image registration toolbox \cite{2009-FAIR} is the backbone of our implementations. We follow FAIR matrix-free approach.

In all the experiments, the regularization-equipped initialization schemes achieve higher accuracy than the simple scaling based approaches, \ie, $\LSp$, $\LSy$, and $\Id$. Moreover, these simple scaling schemes converge to a higher value of the objective function; see TRE and reduction factor column in Tab.~\ref{tab:IR_Results}. 

In terms of TRE, the $\FAIR$ scheme is almost similar to the proposed schemes, but it converges much slower than others; see the run-time column in Tab.~\ref{tab:IR_Results}. The proposed schemes are faster than others in all the experiments but TP-2. Here, $\LSy$ converges faster but at the cost of lower accuracy. It is important to note that, even though the regularization-equipped schemes' per-iteration cost is higher due to the linear solver, they converge faster. It is mainly because of the lower iteration counts, as also seen for quadratic problems; see Tab.~\ref{tab:quadPrb}.

Among the four proposed choices, the $\Dp$ and the $\GM$ turn out to be the best performing schemes. Although, we notice that the performance of a particular scheme greatly depends on the minimizing objective function at hand.

As expected, the steepest descent method is one of the slowest and inaccurate among all. The GN method generally needs fewer iterations, but the per-iteration cost is much higher due to the Jacobian computation; hence the run-time is high.

\end{subsection}
 


\end{section}

%% file: table_quadPrb.tex
\newcolumntype{R}[1]{>{\raggedleft\arraybackslash }b{#1}}
\newcolumntype{L}[1]{>{\raggedright\arraybackslash }b{#1}}
\def\cw{0.87cm}
\pgfplotstableset{mystyle/.style={
		col sep=comma,
		string type,
		columns/LS/.style={column name={}, column type=L{0.48cm}, string type},
		columns/optMethods/.style={column name= {S./$\ell$}, column type=l, string type},
		columns/l1/.style={column name= {$1$}, column type=R{\cw}, string type},
		columns/l3/.style={column name= {$3$}, column type=R{\cw}, string type},
		columns/l5/.style={column name= {$5$}, column type=R{\cw}, string type},
		columns/l10/.style={column name= {$10$}, column type=R{\cw}, string type},
		columns/linf/.style={column name= {$\infty$}, column type=R{\cw}, string type},
		every last row/.style={after row=\bottomrule},
}}


\pgfplotstableset{markstyle/.append style={
		postproc cell content/.append style={
			/pgfplots/table/@cell content/.add={\cellcolor{black!20!white}}{},
		}
}}

\begin{table}[t]
	\caption{
		Optimization results for quadratic problem with eight Hessian initialization strategies (S). Iteration counts and average line-searches (LS) per iteration are mentioned for weakly ($\alpha = 10^{-5}$), mildly ($\alpha = 10^{-3}$), and strongly ($\alpha = 10^{-1}$) regularized problems with $\ell = 1, 5, 10$, and $\infty$. 
	}
	\setlength{\tabcolsep}{0.1pt}
	\renewcommand{\arraystretch}{1}
	\centering
	\begin{minipage}{0.4\columnwidth}
		\begin{filecontents*}{data/quadraticPrb_alpha1e-05_iter.csv}
optMethods,l1,l5,l10,linf
Id,5000,3858,3893,171
LSp,4108,908,383, 37
LSy,2705,1460,870, 90
FAIR,5000,869,389, 18
Dp,4400,817,262, 30
Dz,3716,1128,633, 84
Du,5000,760,274, 30
GM,2880,588,269, 61
\end{filecontents*}
\pgfplotstabletypeset[
		mystyle,
		create on use/LS/.style={
			create col/set list={\multirow{8}{*}{\rotatebox{90}{Iterations}}},
		},
		columns={LS,optMethods,l1,l5,l10,linf},
	    every head row/.style={%
			before row={\toprule 
				& & \multicolumn{4}{c}{$\alpha = 10^{-5}$}\\
				\cmidrule{2-6}},
			after row={
				\midrule},
		},
		]{data/quadraticPrb_alpha1e-05_iter.csv}		
	\end{minipage}
     \begin{minipage}{0.29\columnwidth}
     	\begin{filecontents*}{data/quadraticPrb_alpha1e-03_iter.csv}
optMethods,l1,l5,l10,linf
Id,5000,754,628,128
LSp,539,145, 65, 28
LSy,558,230,115, 39
FAIR,168, 93, 31, 15
Dp,565, 79, 47, 20
Dz,564,228,163, 50
Du,578, 88, 36, 20
GM,356,125, 56, 30
\end{filecontents*}
\pgfplotstabletypeset[
     	mystyle,
     	columns={l1,l5,l10,linf},
     	every head row/.style={%
     		before row={\toprule 
     			\multicolumn{4}{c}{$\alpha = 10^{-3}$}\\
     			\cmidrule{1-4}},
     		after row={
     			\midrule},
     	},
     	]{data/quadraticPrb_alpha1e-03_iter.csv}		
     \end{minipage}
    \begin{minipage}{0.29\columnwidth}
    	\begin{filecontents*}{data/quadraticPrb_alpha1e-01_iter.csv}
optMethods,l1,l5,l10,linf
Id,567,137, 97, 47
LSp, 79, 37, 31, 23
LSy, 69, 39, 30, 23
FAIR, 18,  7,  7,  7
Dp, 23, 12, 10, 10
Dz, 49, 27, 19, 17
Du, 23, 11, 10, 10
GM, 30, 13, 11, 11
\end{filecontents*}
\pgfplotstabletypeset[
    	mystyle,
    	columns={l1,l5,l10,linf},
    	every head row/.style={%
    		before row={\toprule 
    			\multicolumn{4}{c}{$\alpha = 10^{-1}$}\\
    			\cmidrule{1-4}},
    		after row={
    			\midrule},
    	},
    	]{data/quadraticPrb_alpha1e-01_iter.csv}		
    \end{minipage}
	\begin{minipage}{0.40\columnwidth}
		\begin{filecontents*}{data/quadraticPrb_alpha1e-05_avgLS.csv}
optMethods,l1,l5,l10,linf
Id,1.00,1.00,1.00,1.00
LSp,2.73,7.05,8.13,1.16
LSy,1.15,1.10,1.11,1.00
FAIR,14.74,12.92,12.04,5.56
Dp,2.68,6.64,6.86,1.80
Dz,1.15,1.10,1.11,1.18
Du,2.71,6.49,7.55,1.80
GM,1.77,3.04,3.54,1.28
\end{filecontents*}
\pgfplotstabletypeset[
		mystyle,  		
		create on use/LS/.style={
			create col/set list={\multirow{8}{*}{\rotatebox{90}{avg. LS per iter.}}},
		},
		every head row/.style={ 
			output empty row,
		},
		columns={LS,optMethods,l1,l5,l10,linf},
		]{data/quadraticPrb_alpha1e-05_avgLS.csv}		
	\end{minipage}
	\begin{minipage}{0.29\columnwidth}
		\begin{filecontents*}{data/quadraticPrb_alpha1e-03_avgLS.csv}
optMethods,l1,l5,l10,linf
Id,1.00,1.00,1.00,1.00
LSp,2.43,3.78,3.08,1.18
LSy,1.15,1.15,1.17,1.00
FAIR,8.08,6.41,3.74,2.67
Dp,2.32,3.62,2.79,1.65
Dz,1.10,1.06,1.09,1.18
Du,2.38,3.10,2.58,1.65
GM,1.63,1.88,1.89,1.33
\end{filecontents*}
\pgfplotstabletypeset[
		mystyle,  		
		every head row/.style={ 
			output empty row,
		},
		columns={l1,l5,l10,linf},
		]{data/quadraticPrb_alpha1e-03_avgLS.csv}		
	\end{minipage}
	\begin{minipage}{0.29\columnwidth}
		\begin{filecontents*}{data/quadraticPrb_alpha1e-01_avgLS.csv}
optMethods,l1,l5,l10,linf
Id,1.00,1.00,1.00,1.00
LSp,1.76,1.62,1.48,1.43
LSy,1.07,1.05,1.10,1.04
FAIR,2.11,1.43,1.43,1.43
Dp,1.35,1.17,1.20,1.20
Dz,1.04,1.07,1.11,1.12
Du,1.26,1.27,1.20,1.20
GM,1.10,1.15,1.18,1.18
\end{filecontents*}
\pgfplotstabletypeset[
		mystyle,  		
		columns/ITER5/.style={column name= {$\infty$}, column type=R{\cw}, string type},
		every head row/.style={ 
			output empty row,
		},
		columns={l1,l5,l10,linf},
		]{data/quadraticPrb_alpha1e-01_avgLS.csv}		
	\end{minipage}
	\label{tab:quadPrb}
\end{table}

%% file: table_imgReg.tex
\newcolumntype{R}[1]{>{\raggedleft\arraybackslash }b{#1}}
\newcolumntype{L}[1]{>{\raggedright\arraybackslash }b{#1}}


\pgfplotstableset{mystyle/.append style={
		col sep=comma,
		string type,
		columns/distance/.style={column name={}, column type=l, string type},
		columns/optMethods/.style={column name= {M.}, column type=l, string type},
		columns/iter/.style={column name= {iter}, column type=R{0.7cm}, string type},
		columns/feval/.style={column name= {feval}, column type=R{0.7cm}, string type},
		columns/redt/.style={column name= {$\frac{J(\phi)}{J(\phi_0)}$}, column type=R{0.8cm}, string type},
		columns/time/.style={column name= {{time}}, column type=R{1.0cm}, string type},
		columns/speedup/.style={column name= {{speedup}}, column type=R{1.1cm}, string type},
		columns/tre/.style={column name={TRE}, column type=R{1.6cm}, string type},
		create on use/tre/.style={
			create col/assign/.code={%
				\edef\entry{\thisrow{mTRE}  (\thisrow{sTRE})}%
				\pgfkeyslet{/pgfplots/table/create col/next content}\entry
			}
		},
		every last row/.style={after row=\bottomrule},
}}

\pgfplotstableset{markstyle/.append style={
		postproc cell content/.append style={
			/pgfplots/table/@cell content/.add={\cellcolor{black!20!white}}{},
		}
}}


\begin{table}[t]
	\caption{
	Optimization results for four image registration test problems (TP). The iteration counts (iter), the function evaluations (feval), the reduction in objective function $\frac{J(\phi)}{J(\phi_0)}$, the average run-time in seconds, and the mean and standard deviation of TRE are reported.
	The gray-colored cell denotes the \LBFGS~method that achieve either the smallest TRE (higher accuracy) or lowest run-time (faster convergence).
	}
	\setlength{\tabcolsep}{2.5pt}
	\centering
	\begin{minipage}{0.535\columnwidth}
		\begin{filecontents*}{data/TP_Hands_MLIR_SSD_Curv_maxLevel7_alpha1.5e+03_m5_mf1.csv}
optMethods,iter,feval,oracle,time,redt,minJac,maxJac,mTRE,sTRE
SD, 999,1600,2599,23.31,86.27,0.80,1.20,1.04,0.62
LSp,1000,4157,5157,40.46,28.07,0.48,2.00,0.67,0.52
LSy,1000,1029,2029,20.44,24.48,0.52,1.88,0.52,0.29
FAIR,1000,1001,2001,57.10,21.62,0.64,1.86,0.36,0.18
Dp,  53,  59, 112,8.96,21.94,0.64,1.79,0.38,0.17
Dz, 444, 445, 889,71.63,20.49,0.62,2.07,0.37,0.16
Du, 444, 445, 889,71.79,20.49,0.62,2.07,0.37,0.16
GM,  78,  80, 158,13.13,20.84,0.64,1.88,0.35,0.17
GN,  18,  19,  37,4.58,28.45,0.60,1.93,0.69,0.59
\end{filecontents*}
\pgfplotstabletypeset[
		mystyle,
		columns={optMethods,iter,feval,redt,time,tre},
		every head row/.style={%
			before row={\toprule 
				& \multicolumn{5}{c}{TP-1: Hand, SSD, Curvature}\\
				\cmidrule{3-6}},
			after row={
				& & & & \scriptsize (sec.) & \scriptsize mean (std.)\\
				\midrule},
		},
		every row 4 column 4/.style={markstyle},
		every row 7 column 5/.style={markstyle},
		]{data/TP_Hands_MLIR_SSD_Curv_maxLevel7_alpha1.5e+03_m5_mf1.csv}		
	\end{minipage}
	\begin{minipage}{0.455\columnwidth}
		\begin{filecontents*}{data/TP_Hands_MLIR_MI_Elastic_maxLevel7_alpha5.0e-03_m5_mf1.csv}
optMethods,iter,feval,oracle,time,redt,minJac,maxJac,mTRE,sTRE
SD, 169, 270, 439,8.23,84.80,0.59,1.54,0.99,0.63
LSp, 154, 446, 600,9.15,73.53,0.63,1.76,0.64,0.37
LSy, 135, 137, 272,5.74,73.64,0.63,1.72,0.64,0.36
FAIR, 170, 171, 341,12.27,72.78,0.63,1.69,0.56,0.30
Dp,  43,  67, 110,7.63,72.76,0.62,1.70,0.58,0.34
Dz,  72,  94, 166,6.73,72.84,0.62,1.71,0.58,0.32
Du,  43,  67, 110,7.70,72.76,0.62,1.70,0.58,0.34
GM,  64,  66, 130,8.83,72.77,0.63,1.70,0.57,0.31
GN,,,,,,,,,
\end{filecontents*}
\pgfplotstabletypeset[
		mystyle,
		columns={iter,feval,redt,time,tre},
		every head row/.style={%
			before row={\toprule 
				\multicolumn{5}{c}{TP-2: Hand, MI, Elastic}\\
				\cmidrule{1-5}},
			after row={
				& & & \scriptsize (sec.) & \scriptsize mean (std)\\
				\midrule},
		},
	    every row 2 column 3/.style={markstyle},
	    every row 3 column 4/.style={markstyle},
	    every row 8 column 4/.style={postproc cell content/.style={@cell content={}}},
	    every row 8 column 3/.style={postproc cell content/.style={@cell content={$\gg 360$}}},
		]{data/TP_Hands_MLIR_MI_Elastic_maxLevel7_alpha5.0e-03_m5_mf1.csv}		
	\end{minipage}
    	\begin{minipage}{0.535\columnwidth}
    	\begin{filecontents*}{data/TP_DIR4DCT_MLIR_NGF_Curv_Case1_maxLevel5_alpha1.0e+02_m5_mf1.csv}
optMethods,iter,feval,oracle,time,redt,minJac,maxJac,mTRE,sTRE
SD, 145, 251, 396,145.83,97.17,0.85,1.17,3.64,2.70
LSp,  75, 204, 279,87.44,95.93,0.76,1.44,2.94,2.20
LSy, 184, 185, 369,163.00,94.84,0.70,1.42,1.70,0.99
FAIR, 127, 128, 255,170.31,94.67,0.69,1.41,1.59,0.79
Dp,  58,  87, 145,128.58,94.79,0.68,1.42,1.61,0.83
Dz, 131, 132, 263,161.70,94.79,0.68,1.41,1.64,0.88
Du, 148, 149, 297,183.07,94.74,0.68,1.42,1.61,0.80
GM,  77,  79, 156,123.45,94.76,0.68,1.42,1.62,0.83
GN,  68, 254, 322,2616.43,94.90,0.68,1.42,1.68,0.99
\end{filecontents*}
\pgfplotstabletypeset[
    	mystyle,
    	columns={optMethods,iter,feval,redt,time,tre},
    	every head row/.style={%
    		output empty row,
    		before row={
    			& \multicolumn{5}{c}{TP-3: Lung, NGF, Curvature}\\
    			\cmidrule{2-6}},
    	},
    	every row 7 column 4/.style={ markstyle},
    	every row 3 column 5/.style={ markstyle},
    	]{data/TP_DIR4DCT_MLIR_NGF_Curv_Case1_maxLevel5_alpha1.0e+02_m5_mf1.csv}		
    \end{minipage}
    \begin{minipage}{0.455\columnwidth}
    	\begin{filecontents*}{data/TP_disc2C_MLIR_SSD_Hyper_maxLevel4_alpha1.0e+00_m5_mf1.csv}
optMethods,iter,feval,oracle,time,redt,minJac,maxJac,mTRE,sTRE
SD, 999,1585,2584,10.15,25.96,0.24,17.08,,
LSp, 778,3520,4298,11.33,16.90,0.31,8.02,,
LSy, 671, 829,1500,6.03,16.93,0.31,8.09,,
FAIR, 214,1083,1297,7.99,6.03,0.36,8.64,,
Dp, 117, 695, 812,4.22,6.02,0.36,8.55,,
Dz, 148, 656, 804,3.90,6.02,0.37,8.57,,
Du, 134, 535, 669,3.31,6.07,0.37,8.61,,
GM, 213, 384, 597,4.35,16.85,0.31,8.05,,
GN,  60, 100, 160,12.08,8.30,0.39,8.39,,
\end{filecontents*}
\pgfplotstabletypeset[
    	mystyle,
    	columns/blank/.style={column name= {temp}, column type=R{1.7cm}, string type},
    	create on use/blank/.style={
    		create col/assign/.code={%
    			\edef\entry{{}}%
    			\pgfkeyslet{/pgfplots/table/create col/next content}\entry
    		}
    	},
    	columns={iter,feval,redt,time,blank},
    	every head row/.style={%
    		output empty row,
    		before row={
    			\multicolumn{5}{c}{TP-4: Disc-C, SSD, Hyperelastic}\\
    			\cmidrule{1-5}},
    	},
    	every row 6 column 3/.style={ markstyle},
    	]{data/TP_disc2C_MLIR_SSD_Hyper_maxLevel4_alpha1.0e+00_m5_mf1.csv}	
    \end{minipage}
	\label{tab:IR_Results}
\end{table}

%% file: 04-conclusion.tex
\begin{section}{Conclusion}\label{Sec:conclusion}
	
We have proposed a Hessian initialization strategy particularly suited for large-scale non-linear inverse problems. 
Typically, the objective function is the sum of a data-fidelity term and a regularizer.
Often, the Hessian of the data-fidelity is computationally expensive. But not the Hessian of the regularizer.

We propose to replace the Hessian of the data-fidelity with a scalar and keep the Hessian of regularizer to initialize the Hessian approximation at every iteration. The scalar satisfies the well-known secant equation in the sense of ordinary and total least squares, and geometric mean regression. In total, we have proposed four choices for the scalar that leads to well-scaled search directions. We also established the inter-relationship between the derived scalars and discussed the consequences of a scalar choice on the convergence in terms of iteration counts and line-search steps. The implementation of our strategy requires only a small change of a standard \LBFGS~code.

Our experiments on highly non-convex image registration problems indicate that the proposed schemes converge faster and achieve higher accuracy than the simple scaling based approaches. The $\Dp$, based on ordinary least squares, and $\GM$, based on geometric mean regression, are best-performing schemes.

Under suitable assumptions, we can also show that the proposed parameters are the eigenvalue's estimates of the Hessian of a data-fidelity term. The theoretical investigation will be a part of the extended version of this paper. Future work also addresses the application to inverse problems, \eg, ultrasound tomography \cite{Bernhardt2020}, and optical tomography \cite{Saratoon2013}.

\end{section}


%% file: main.bbl
\begin{thebibliography}{10}
\providecommand{\url}[1]{\texttt{#1}}
\providecommand{\urlprefix}{URL }
\providecommand{\doi}[1]{https://doi.org/#1}

\bibitem{Andrei2020}
Andrei, N.: A new accelerated diagonal {Quasi-Newton} updating method with
  scaled forward finite differences directional derivative for unconstrained
  optimization. Optimization pp. 1--16 (Jan 2020).
  \doi{10.1080/02331934.2020.1712391}

\bibitem{Bernhardt2020}
Bernhardt, M., Vishnevskiy, V., Rau, R., Goksel, O.: Training variational
  networks with multidomain simulations: Speed-of-sound image reconstruction.
  {IEEE} Transactions on Ultrasonics, Ferroelectrics, and Frequency Control
  \textbf{67}(12),  2584--2594 (Dec 2020). \doi{10.1109/tuffc.2020.3010186}

\bibitem{Burger2013}
Burger, M., Modersitzki, J., Ruthotto, L.: A hyperelastic regularization energy
  for image registration. {SIAM} Journal on Scientific Computing
  \textbf{35}(1),  B132--B148 (Jan 2013). \doi{10.1137/110835955}

\bibitem{Castillo2009}
Castillo, R., Castillo, E., Guerra, R., Johnson, V.E., McPhail, T., Garg, A.K.,
  Guerrero, T.: A framework for evaluation of deformable image registration
  spatial accuracy using large landmark point sets. Physics in Medicine and
  Biology  \textbf{54}(7),  1849--1870 (Mar 2009).
  \doi{10.1088/0031-9155/54/7/001}

\bibitem{Dener2019}
Dener, A., Munson, T.: Accelerating limited-memory {Quasi-Newton} convergence
  for large-scale optimization. In: LNCS, pp. 495--507. Springer (2019).
  \doi{10.1007/978-3-030-22744-9-39}

\bibitem{Draper1991}
Draper, N.R.: Straight line regression when both variables are subject to
  error. Conference on Applied Statistics in Agriculture  (1991).
  \doi{10.4148/2475-7772.1414}

\bibitem{Gilbert1989}
Gilbert, J.C., Lemar{\'{e}}chal, C.: Some numerical experiments with
  variable-storage {Quasi-Newton} algorithms. Mathematical Programming
  \textbf{45}(1-3),  407--435 (Aug 1989). \doi{10.1007/bf01589113}

\bibitem{Hansen2010}
Hansen, P.C.: Discrete Inverse Problems. Society for Industrial and Applied
  Mathematics (Jan 2010). \doi{10.1137/1.9780898718836}

\bibitem{Heldmann2006}
Heldmann, S.: Non-linear registration based on mutual information theory,
  numerics, and application. Logos-Verl, Berlin (2006)

\bibitem{Jiang2004}
Jiang, L., Byrd, R.H., Eskow, E., Schnabel, R.B.: {Preconditioned L-BFGS
  Algorithm with Application to Molecular Energy Minimization}. Tech. rep.,
  Colorado Univ. at Boulder. Dept. of Computer Science (2004)

\bibitem{Knig2018}
K\"{o}nig, L., R\"{u}haak, J., Derksen, A., Lellmann, J.: A matrix-free
  approach to parallel and memory-efficient deformable image registration.
  {SIAM} Journal on Scientific Computing  \textbf{40}(3),  B858--B888 (Jan
  2018). \doi{10.1137/17m1125522}

\bibitem{Liu1989}
Liu, D.C., Nocedal, J.: On the limited memory {BFGS} method for large scale
  optimization. Mathematical Programming  \textbf{45}(1-3),  503--528 (Aug
  1989). \doi{10.1007/bf01589116}

\bibitem{Marjugi2013}
Marjugi, S.M., Leong, W.J.: Diagonal {Hessian} approximation for limited memory
  {Quasi-Newton} via variational principle. Journal of Applied Mathematics
  \textbf{2013}, ~1--8 (2013). \doi{10.1155/2013/523476}

\bibitem{2009-FAIR}
Modersitzki, J.: {FAIR}: Flexible Algorithms for Image Registration. SIAM,
  Philadelphia (2009)

\bibitem{Nocedal2006}
Nocedal, J., Wright, S.J.: Numerical Optimization. Springer (2006).
  \doi{10.1007/978-0-387-40065-5}

\bibitem{Oren1982}
Oren, S.S.: Perspectives on self-scaling variable metric algorithms. Journal of
  Optimization Theory and Applications  \textbf{37}(2),  137--147 (Jun 1982).
  \doi{10.1007/bf00934764}

\bibitem{Patil2019}
Patil, N., Naik, N.: Second-order adjoint sensitivities for fluorescence
  optical tomography based on the {SPN} approximation. Journal of the Optical
  Society of America A  \textbf{36}(6), ~1003 (May 2019).
  \doi{10.1364/josaa.36.001003}

\bibitem{Saratoon2013}
Saratoon, T., Tarvainen, T., Cox, B.T., Arridge, S.R.: A gradient-based method
  for quantitative photoacoustic tomography using the radiative transfer
  equation. Inverse Problems  \textbf{29}(7),  075006 (Jun 2013).
  \doi{10.1088/0266-5611/29/7/075006}

\bibitem{Vogel2002}
Vogel, C.R.: Computational Methods for Inverse Problems. Society for Industrial
  and Applied Mathematics (Jan 2002). \doi{10.1137/1.9780898717570}

\end{thebibliography}
